\documentclass[12pt]{amsart}

\usepackage{amsmath}
\usepackage{amssymb}
\usepackage{amsfonts}
\usepackage{amsthm}
\usepackage{enumerate}
\usepackage{hyperref}
\usepackage{color}
\usepackage{psfrag}
\usepackage[all]{xy}
\usepackage{tikz}

\theoremstyle{plain}

\newtheorem{thm}{Theorem}[section]

\newtheorem{lem}[thm]{Lemma}
\newtheorem{cor}[thm]{Corollary}

\theoremstyle{definition}

\newtheorem{rem}[thm]{Remark}

\newtheorem{dfns-rems}[thm]{Definitions and Remarks}
\newtheorem{notas-rems}[thm]{Notations and Remarks}
\newtheorem{exmps-rems}[thm]{Examples and Remarks}

\begin{document}


\title[On The Double Roman ON bondage]{On The Double Roman bondage numbers of Graphs}
\author[N. Jafari Rad, H.R. Maimani, M. Momeni and F. Rahimi Mahid ]{}

\author[]{N. Jafari Rad}
\address{N. Jafari Rad, Department of Mathematics,
Shahed University, Tehran, Iran}

\email{n.jafarirad@gmail.com}
\author[]{H.R. Maimani}
\address{H.R. Maimani, Mathematics Section, Department of Basic Sciences,
Shahid Rajaee Teacher Training University, P.O. Box 16785-163,
Tehran, Iran}

\email{maimani@ipm.ir}

\author[]{M. Momeni}
\address{M. Momeni,  Mathematics Section, Department of Basic Sciences,
Shahid Rajaee Teacher Training University, P.O. Box 16785-163,
Tehran, Iran}

\email{momeni.mosi@yahoo.com}

\author[]{F. Rahimi Mahid}

\address{F. Rahimi Mahid, Mathematics Section, Department of Basic Sciences,
Shahid Rajaee Teacher Training University, P.O. Box 16785-163,
Tehran, Iran}

\email{farhad.rahimi@sru.ac.ir}

\date{}

\begin{abstract}
 For a graph $G=(V,E)$, a double roman dominating function (DRDF)
is a function $f : V \longrightarrow \{0, 1, 2,3\}$ having the
property that if $f(v)=0$ for some vertex $v$, then $v$ has at
least two neighbors assigned $2$ under $f$ or one neighbor $w$
with $f(w)=3$, and if $f(v)=1$ then $v$ has at least one neighbor
$w$ with $f(w) \geq 2$. The weight of a DRDF $f$ is the sum  $f
(V) =\sum_{u\in V} f (u)$. The minimum weight of a DRDF on a
graph $G$ is the double Roman domination number of $G$ and is
denoted by $\gamma_{dR}(G)$. The double roman bondage number of
$G$, denoted by $b_{dR}(G)$, is the minimum cardinality among all
edge subsets $B \subseteq E(G)$ such that $\gamma_{dR}(G-B) >
\gamma_{dR}(G)$. In this paper we study the double roman
bondage number in graphs. We determine the double roman
bondage number in several families of graphs, and present several bounds for the double roman
bondage number. We also study the complexity issue of the double roman
bondage number and prove that the decision problem for the double roman
bondage number is NP-hard even when restricted to bipartite graphs.
\end{abstract}
\subjclass[2010]{05C69} \keywords{Domination, Double Roman
Domination, Bondage number, Planar graph, 3-SAT, NP-hard.}

\maketitle
\section{Introduction}
Throughout this paper all graphs are finite, simple and
undirected. We denote the vertex set and the edge set of a graph
$G$ by $V=V(G)$ and $E=E(G)$, respectively. \\Let $G = (V, E)$ be
a graph of order $n$. The \textit{open neighborhood} of a vertex
$v\in{V}$ is the set $N(v) = \{u|uv \in{E}\},$ and the
\textit{closed neighborhood} of $v$ is $N[v] = N(v) \cup \{v\}$.
The \textit{degree} of a vertex $v$ is $\deg_G(v)=|N(v)|$. The
\textit{maximum} (respectively, \textit{minimum}) degree among
the vertices of $G$ is denoted by $\Delta(G)$ (respectively,
$\delta(G)$). The open neighborhood of a set $S \subseteq V$ is
$N(S) = \cup_{ v\in{S}} N(v),$ and the closed neighborhood of $S$
is $N[S] = N(S)\cup S = \cup_{v\in{S}} N[v].$ A vertex with
exactly one neighbor is called a \textit{leaf} and its unique
neighbor is a \textit{support vertex}. A \textit{strong support
vertex} is a support vertex adjacent to at least two leaves. The
distance $d_G(x,y)$ (or briefly $d(x,y)$) between vertices $x$
and $y$ of a graph $G$ is the length of shortest path connecting
them. The \textit{girth} $g(G)=g$ of $G$ is the length of a
shortest cycle in $G$, and $g(G)=\infty $ when $G$ is a forest. A
\textit{$k$-partite graph} is a graph which its vertex set can be
partitioned into $k$ sets $V_1, V_2, \cdots, V_k$
 such that every edge of the graph has an end point in $V_i$
and an end point in $V_j$ for some $1\leq i\neq j\leq k$. A
\textit{complete $k$-partite graph} is a $k$-partite graph that
every vertex of each partite set is adjacent to all vertices of
the other partite sets. We denote by $K_{n_1,n_2,\cdots n_k}$ the
complete $k$-partite graph where $|V_i|=n_i$ for $1\leq i\leq k$.
In the case $k=2$, the $k$-partite and complete $k$-partite graph
are called bipartite and complete bipartite graphs. We denote by
$P_n, C_n, K_n$ and $\overline{K_n}$, the path, the cycle, the
complete graph and the empty graph of order $n$, respectively.
For a graph $G$ and a nonempty subset $S\subseteq V(G)$, the
\textit{vertex-induced subgraph}, denoted by $G[S]$,
is the subgraph of $G$ with vertex-set $S$ and edges incident to
members of $S$. For a subset $S$ of vertices, we refer to $G-S$
as the subgraph of $G$ induced by $V(G)\setminus S$. If
$S=\{v\},$ then the subgraph $G-S$ is denoted by $G-v$. For a
nonempty subset $X\subseteq E(G)$, we denote by $G-X$ the
spanning subgraph of $G$ obtained by deleting the edges of $X$
from $G$. If $X=\{e\}$, then we denote $G-X$ by $G-e$. A
\textit{planar graph} is a graph that can be drawn on the plane
in such a way that its edges intersect only at their endpoints. A
connected graph $G$ is called $2$-connected, if for every vertex
$x \in  V (G)$, the graph $G-x$ is connected. The \textit{join} of
two graphs $G$ and $H$,  $G\vee H$, is the graph with vertex-set
$V(G\vee H)=V(G)\cup V(H)$ and edge set $E(G\vee H)=E(G)\cup
E(H)\cup\{uv: u\in V(G), v\in V(H)\}.$ The graph $K_1\vee
C_{n-1}$ is called a wheel and is denoted by $W_n$.

A set $S \subseteq V$ is called a \textit{dominating set} if
$N[S] = V.$ The \textit{domination number}, $\gamma (G)$ of $G$,
is the minimum cardinality of a dominating set in $G$. A
dominating set of $G$ of cardinality $\gamma (G)$ is called a
$\gamma$-set of $G$ or just a $\gamma(G)$-set. For other graph
theory notation and terminology not given here we refer to
\cite{hhs}.

Let $f : V \rightarrow\{0, 1, 2\}$ be a function having the
property that for every vertex $v\in{V}$ with $f(v) = 0,$ there
exists a neighbor $u \in {N(v)}$ with $f(u) = 2.$ Such a function
is called a \textit{Roman dominating function}. The weight of a
Roman dominating function is the sum $f(V) = \sum_{v\in{V}}
f(v).$ The minimum weight of a Roman dominating function on $G$
is called the \textit{Roman domination number} of $G$ and is
denoted by $\gamma_{R}(G)$. A Roman dominating function on $G$ of
weight $\gamma_{R}(G)$ is called a $\gamma_{R}$-function of $G$
or just a $\gamma_R(G)$-function. A Roman dominating function $f$
can be represented as a triple $f=(V_0,V_1,V_2)$ (or
$f=(V_0^f,V_1^f,V_2^f)$), where $V_i=\{v:f(v)=i\}$ for $i=0,1,2$.
The mathematical concept of Roman domination, defined and
discussed by Stewart \cite{s}, and ReVelle and Rosing \cite{rr},
and subsequently developed by Cockayne et al. \cite{cdhh}. Several
variants of Roman domination already have been defined and
studied. For a graph $G=(V,E)$, a \textit{double roman dominating
function} (DRDF) is a function $f : V \longrightarrow \{0, 1,
2,3\}$ having the property that if $f(v)=0$, then $v$ has at
least two neighbors assigned $2$ under $f$ or one neighbor $w$
with $f(w)=3$, and if $f(v)=1$ then $v$ has at least one neighbor
$w$ with $f(w) \geq 2$. The weight of a double Roman dominating
function $f$ is the sum  $f (V) =\sum_{u\in V} f (u)$. The
minimum weight of a DRDF is called \textit{double Roman
domination number} of $G$ and is denoted by $\gamma_{dR}(G)$. The
concept of double roman domination is defined by Beeler,  Haynes
and Hedetniem \cite{bhh}, and further studied in, for example
\cite{acs, jr,v}. Beeler et al. \cite{bhh} observed that in a
DRDF of minimum weight no vertex needs to be assigned the value
$1$. In fact for every DRDF $f : V \longrightarrow \{0,1, 2,
3\}$, there is a DRDF $f' : V \longrightarrow \{0, 2, 3\}$ with
$w(f')\leq w(f)$. Thus, since $\gamma_{dR}(G)$ is the minimum
weight among all double Roman dominating functions on $G$,
without loss of generality, we only consider double Roman
dominating functions with no vertex assigned $1$. We use the
notation $f=(V^f_0,V^f_2,V^f_3)$ for a DRDF $f : V
\longrightarrow \{0, 2, 3\}$, where $V_i^f=\{v:f(v)=i\}$, for
$i=0,2,3$.

Bauer, Harary, Nieminen and Suffel \cite{bhns} introduced the
concept of bondage number in graphs. The \textit{bondage number}
$b(G)$ of a nonempty graph $G$ is the minimum cardinality among
all sets of edges $E'\subseteq E(G)$ for which
$\gamma(G-E')>\gamma(G)$. This concept has been further studied
for various domination variants, see for example, \cite{akq,bhsx,
dhtv,fjkr,hr,jv,x}.

In this paper we consider the concept of bondage number for the
double roman domination number. The \textit{double roman bondage
number} of a graph $G$, denoted by $b_{dR}(G)$, is the minimum
cardinality among all edge subsets $B \subseteq E(G)$ such that
$\gamma_{dR}(G-B) > \gamma_{dR}(G)$. The organization of the
paper is as follows. In Section 2, we present some preliminary
results and determine the double Roman bondage number in some
families of graphs. In Section 3, we present various bounds for
the double Roman bondage number. In Section 4, we study
complexity issue of the double Roman bondage number, and show
that the decision problem of the double Roman bondage number is
NP-hard even when restricted to bipartite graphs. We make use of
the following.

\begin{thm}\cite{acs} \label{12}
Let $G$ be a connected  graph of order $n \geq 3$. Then
\begin{itemize}
\item[1.]
$\gamma_{dR}(G) = 3$ if and only if $\Delta (G) = n -1$.
\item[2.]
$\gamma_{dR}(G) = 4$ if and only if $G = \overline{K_2} \vee H$, where $H$ is a graph with $\Delta (H) \leq |V(H)|-2$.
\item[3.]
$\gamma_{dR}(G) = 5$ if and only if $\Delta (G) = n - 2$ and $G \neq \overline{K_2} \vee H$ for any graph $H$ of order $n-2$.
\end{itemize}
\end{thm}

\begin{thm}\cite{acs}\label{pncn}
Let $n$ be a positive integer. Then
\[
\gamma_{dR}(P_n)=\left\{
    \begin{array}{ll}
      n, & \hbox{if $n\equiv 0 \, (mod \, 3)$}, \\
      n+1, & \hbox{if $n\not\equiv 0 \, (mod \, 3)$.}

    \end{array}
  \right.
\]\\
\[
\gamma_{dR}(C_n)=\left\{
    \begin{array}{ll}
      n+1, & \hbox{if $n\equiv 1, 5 \, (mod\, 6)$,} \\
      n, & \hbox{if $n\not\equiv 1, 5 \, (mod\, 6)$.}

    \end{array}
  \right.
\]
\end{thm}

\begin{lem}\label{1.3} \cite{v}
Let $1\leq n_1\leq n_2\leq \cdots\leq n_r$ be integers. Then
\[
\gamma_{dR}(K_{n_1,n_2,\cdots,n_r})=\begin{cases}
3 \quad if\ n_1=1,\\
4\quad if \ n_1=2,\\
6\quad  \ n_1\geq3.\\
\end{cases}
\]
\end{lem}

\begin{thm}\cite{w}\label{9}
If $G$ is a planar graph of girth $g\leq \infty$, then
\[
|E(G)|\leq \frac{g(n(G)-2)}{g-2}.
\]
\end{thm}

\begin{cor}\cite{w}\label{15}
If $G$ is a planar graph of order $n(G)\geq 3$. then
\[
|E(G)|\leq 3n(G)-6.
\]
\end{cor}

\begin{cor}\cite{w}\label{10}
If $G$ is a planar graph, then $\delta(G)\leq 5$. If $G$ is a
planar graph of girth $g\geq 4$, then $\delta(G)\leq 3$. If $G$
is a planar graph of girth $g\geq 6$, then $\delta(G)\leq 2$.
\end{cor}
\begin{lem}\cite{w}\label{16}
Let $v$ be a vertex of a planar graph $G$ with $\deg_G(v)\geq 3$,
and let $E_v=\{xy|x, y \in N(v), xy \notin E(G)\}$. Then there
exists a subset $S\subseteq E_v$ such that $H=G+S$ is still a
planar graph and $H[N(v)]$ is $2$-connected.
\end{lem}

\section{Preliminaries and Exact values}

In this section we present some preliminary results, and determine
the double Roman bondage number in several families of graphs
including paths, cycles, complete graphs and complete bipartite
graphs. We first determine the double bondage number of paths
$P_n$.

\begin{thm}\label{p_n}
For any $n\geq 1$, $b_{dR}(P_n)=1$.
\end{thm}

\begin{proof}
Let $V(P_n)=\{v_1,v_2\ldots v_n\}$. If $n=3k$, then by Theorem
\ref{pncn}, $\gamma_{dR}(P_n- v_2v_3)=\gamma_{dR}(P_2\cup
P_{n-2})=n+2>\gamma_{dR}(P_n)$ and so $b_{dR}(P_n)=1$. If
$n=3k+1$, then by Theorem \ref{pncn},
$\gamma_{dR}(P_n-v_2v_3)=\gamma_{dR}(P_2\cup
P_{n-2})=3+n-2+1=n+2>\gamma_{dR}(P_n)$ and so $b_{dR}(P_n)=1$.
Now assume that $n=3k+2$. Then by Theorem \ref{pncn},
$\gamma_{dR}(P_n-v_1v_2)=\gamma_{dR}(P_1\cup
P_{n-1})=2+n-1+1=n+2>\gamma_{dR}(P_n)$ and so $b_{dR}(P_n)=1$.
\end{proof}

\begin{thm}
For any $n\geq 3$ we have:\\
\[
b_{dR}(C_n)=\left\{
    \begin{array}{ll}
      1, & \hbox{if $n\equiv 2$ or $4\, (mod\, 6)$;} \\
      2, & \hbox{if $n\not\equiv 2$ or $4\, (mod\, 6)$.}

    \end{array}
  \right.
\]
\end{thm}

\begin{proof}
Let $V(C_n)=\{v_1,v_2\ldots v_n\}$. Clearly removing any vertex
of $C_n$ leaves a $P_n$. If $n \equiv 2$ or $4$ mod $6$, then
according to Theorem \ref{pncn}, $\gamma_{dR}(C_n)=n$, while
$\gamma_{dR}(P_n)=n+1$, and so $b_{dR}(C_n)=1$. Next assume that
$n \not \equiv 2$ or $4$ mod $6$. Then Theorem \ref{pncn} leads
to $\gamma_{dR}(C_n)=\gamma_{dR}(P_n)$. We thus obtain that
$b_{dR}(C_n) \geq 2$. On the other hand because of $b_{dR}(C_n)
\geq 2$, for each edge $e_1\in E(C_n)$ we have
$\gamma_{dR}(C_n)=\gamma_{dR}(C_n-e_1)=\gamma_{dR}(P_n)$. By
Theorem \ref{p_n}, there exists an edge $e_2\in E(C_n)$ such that
$\gamma_{dR}(C_n-\{e_1,e_2\})=\gamma_{dR}(P_n-e_2)>\gamma_{dR}(P_n)=\gamma_{dR}(C_n)$,
thus $b_{dR}(C_n) \leq 2$. Consequently, $b_{dR}(C_n) = 2$.
\end{proof}

\begin{cor}\label{n1}
If $G$ is a graph of order $n \geq 3$ with exactly $k \geq 1$ vertices of degree $n-1$, then $b_{dR}(G) = \lceil k/2 \rceil$.
\end{cor}

\begin{proof}
Since $k \geq 1$, we have $\gamma_{dR}(G)=3$ by Theorem \ref{12}.
Let $S=\{v\in G, \deg_G(v)=n-1\}$. Then $S$ is a clique of $G$.
Consider a minimum edge cover $E^\prime$ of $S$. We know that
$|E^\prime |=\lceil  \frac{k}{2} \rceil$. Now $G - E^\prime$ has
no vertex of degree $n-1$, and so $\gamma_{dR}(G -
E^\prime)>3=\gamma_{dR}(G)$. Thus we conclude that $b_{dR}(G)
\leq \lceil \frac{k}{2} \rceil$.

Now suppose that $b_{dR}(G)=l,$ and let $S$ be an edge set of
size $l$ such that $\gamma_{dR}(G - S)>\gamma_{dR}(G)$.
Note that $S$ covers at most $2l$ vertices of $G$. If $2l<k$,
then $G-S$ has at least one vertex of degree $n-1$, and
hence $\gamma_{dR}(G- S)=3=\gamma_{dR}(G)$, a
contradiction. Therefore $2l \geq k$, and we conclude that
$b_{dR}(G) \geq \lceil \frac{k}{2} \rceil$.
\end{proof}

As a consequence of Corollary \ref{n1}, we have the following.

\begin{cor}
If $n \geq 3$, then $b_{dR}(K_n) = \lceil n/2 \rceil$, $b_{dR}(W_n)=1.$
\end{cor}

We next determine the double bondage number of complete multi partite graphs.

\begin{lem}
Let $1\leq n_1\leq n_2\leq \cdots\leq n_r$ be integers. Then

\[
b_{dR}(K_{n_1,n_2,\cdots,n_r})=\left\{
    \begin{array}{llll}
\lceil \frac{l}{2}\rceil &   \hbox{if $  n_1=\cdots =n_l=1, n_{l+1}\geq 2$}  \\
\lceil \frac{l}{2}\rceil & \hbox{if $ n_1=\cdots =n_l=2, n_{l+1}\geq 3$}  \\
3(r-1)+1 &  \hbox{if $ n_1=n_2=\cdots =n_r=3$} \\
\sum_{i=1}^{r-1}n_i &   \hbox{if  otherwise}.
   \end{array}
  \right.
\]
\end{lem}

\begin{proof}
Let $G=K_{n_1, n_2,\cdots, n_r}$ and $V_1, V_2, \cdots, V_r$ be
partite sets of $V(G)$, where $|V_j|=n_j$ for $1\leq j\leq r$.
Also suppose that $X\subseteq E(G)$ such that $\gamma_{dR}(G)<
\gamma_{dR}(G- X)$. Let $H$ be the subgraph of $G$ induced by $X$ and let
$T=V(H)$.\\ If $|V_1|=|V_2|=\cdots =|V_l|=1$ and $|V_{l+1}|\geq
2$, then $b_{dR}(K_{n_1,n_2,\cdots,n_r})=\lceil\frac{l}{2}\rceil$
by Corollary \ref{n1}. Now suppose that $|V_1|=|V_2|=\cdots
=|V_l|=2$ and $|V_{l+1}|\geq 3$. If $T\cap V_i= \varnothing$
for some $1\leq i\leq l$, then $G-
X\cong\overline{K_2}\vee K$ for some graph $K$ such that $\Delta(K)\leq |V(K)|-2$. Hence
$\gamma_{dR}(G)=\gamma_{dR}(G- X)=4$ by Theorem \ref{12}, which is
a contradiction. Hence $T\cap V_i\neq \varnothing$ for each
$1\leq i\leq l$. Then $|T|\geq l$. Therefore $|X|\geq
\lceil\frac{l}{2}\rceil$ and we conclude that $b_{dR}(G)\geq
\lceil\frac{l}{2}\rceil$. For $1\leq i\leq l$, choose $v_i\in
V_i$ and consider $Y=\{v_1v_2, v_3v_4, \cdots v_{l-1}v_l\}$ if
$l$ is even and $Y=\{v_1v_2, v_3v_4, \cdots v_{l-2}v_{l-1},
v_1v_l\}$ if $l$ is odd. In both cases $G- Y\ncong
\overline{K_2}\vee K$ for each graph $K$. Hence
$\gamma_{dR}(G-Y)\geq 5$ by Theorem \ref{12} and therefore
$b_{dR}(G)\leq \lceil\frac{l}{2}\rceil$.

Now suppose that $n_i\geq 3$, for any $1\leq i \leq r$. In this
case $\gamma_{dR}(G)=6$ by Lemma \ref{1.3}. Therefore
$V_i\subseteq T$ for each $i$, except probably one $i$, since if
$v_t\in V_t\setminus T$ for $t\in \{t_1, t_2\}$, then $\{V_{t_1},
V_{t_2}\}$ dominate all vertices of $G- X$ and hence
$\gamma_{dR}(G- X)\leq 6$, which is a contradiction. We consider
two cases.

{\bf Case 1.} For each $1\leq i\leq r$, $n_i=3$.\\
Suppose that $\bigcup_{i=1}^{r-1}V_i\subseteq T$. Then $|V_r\cap T|\geq 2$, since otherwise $\gamma_{dR}(G- X)=6$. First suppose that $|V_r\cap T|=2$. Assume that $V_r=\{y,y',y''\}$
and $y\in V_r\setminus T$. If there exists $z\in \bigcup_{i=1}^{r-1}V_i$ such that $z$ is adjacent to $y'$ and $y''$ in $G-X$, then $(V\setminus \{z,y\},\emptyset , \{z,y\})$ is a DRDF, and hence
$\gamma_{dR}(G- X)\leq 6$, which is a contradiction. So for any $z\in \bigcup_{i=1}^{r-1}V_i$, there exists an edge $zy'$ or $zy''$ which belongs to $X$. On the other hand there exists 
$x\in \bigcup_{i=1}^{r-1}V_i$, which is not adjacent to $y',y''$ in $G-X$, since otherwise $(V\setminus \{y,y',y''\},\{y,y',y''\},\emptyset)$ is a DRDF for $G-X$, which is a contradiction. Hence 
$|X|\geq 3(r-1)+1$. Now suppose that $V_r\subseteq T$. Hence $H$ is a spanning subgraph of $G$. If for any vertex $x\in V(G)$ we have $\deg_H(x)\geq 2$, then
\[
2|X|= \sum_{x\in V}\deg_H(x)\geq 2(3r)
\]

and so $|X|\geq 3r>3(r-1)+1.$ Suppose that there exists a vertex $x\in V(G)$, with $\deg_H(x)=1$. Without lose of generality, suppose that $x\in V_1=\{x,x',x''\}$ and $x$ is 
adjacent to $y\in V_r$. If $y$ is adjacent to $x'$ and $x''$ in $G- X$, then $(V\setminus \{x,y\}, \emptyset , \{x,y\})$ is a DRDF for $G- X$, which is a contradiction. Hence $yx'$ or $yx''$ 
belongs to $X$. Also there are two edges $y't_1,y''t_2\in X$ for some vertices $t_1, t_2$ of $G$. If there exists a vertex $z\in \bigcup_{i=2}^{r-1}V_i$, such that $z$ is adjacent to all vertices $y, x', x''$ 
in $G- X$, then $(V\setminus \{z,x\},\emptyset, \{z,x\})$ is a DRDF for $G- X$, which is impossible. Hence for any $z\in \bigcup_{i=2}^{r-1}V_i$, $X\cap \{zx',zx'',zy\}\neq \emptyset$. Therefore
\[
|X|\geq 3(r-2)+2+2=3(r-1)+1.
\]

Hence in this case $b_{dR}(G)\geq 3(r-1)+1$. On the other hand for $x,x'\in V_1, y\in V_2$ consider the set $X=\{xz ; z\in \bigcup_{i=2}^{r}V_i, x'y\}$. Clearly $G- X=K_1\cup K$, and
$K\ncong \overline{K_2}\vee L$ for any graph $L$. This means that $\gamma_{dR}(G- X)\geq 2+5=7$, and hence $b_{dR}(G)\leq 3(r-1)+1$. We conclude that $b_{dR}(G)= 3(r-1)+1$.

{\bf Case 2.} For at least one $i$, $n_i\geq 4$. The argument of this case is similar to case 1.
\end{proof}

\section{Bounds for the double Roman bondage number}
In this section we present various bounds for the double Roman bondage number. We begin with the following.

\begin{thm}\label{thmn1}
If $G$ is a graph, and $xyz$ a path of length $2$ in $G$, then \\
$(1)$ $b_{dR}(G) \leq \deg_G(x)+\deg_G(y)+\deg_G(z)-3- \vert N(x) \cap N(y) \vert$.\\
If $x$ and $z$ are adjacent, then\\
$(2)$ $b_{dR}(G) \leq \deg_G(x)+\deg_G(y)+\deg_G(z)-4- \vert N(x) \cap N(y) \vert$.\\
\end{thm}
\begin{proof}
Let $H$ be the graph obtained from $G$ by removing the edges
incident to $x,y$ and $z$ with exception of $yz$ and all edges
between $y$ and $N(x) \cap N(y)$. In $H$, the vertex $x$ is
isolated, $z$ is leaf, $y$ is adjacent to $z$, and all neighbors
of $y$ in $H$, if any, lie in $N_G(x)$.

Let $f=(V_0,V_2,V_3) $ be a $\gamma_{dR}(H)$-function. Then $x
\in V_2$ and, without loss of generality, assume that $z \in V_0
\cup V_2$. If $z \in V_0$, then $y \in V_3$ and therefore $(V_0
\cup \{ x \},V_2\setminus \{ x \},V_3)$ is a DRDF on $G$ of
weight less than $f$, and $(1)$ as well as $(2)$ are proved.

Now assume that $z \in V_2$. Obviously $y \not \in V_3$. If $y \in
V_2$, then $(V_0 \cup \{ z \}, V_2\setminus \{ y,z \},V_3 \cup \{
y \})$ is a DRDF on $H$ of weight less than $f$, that is a
contradiction. However, if $y \in V_0$, then $((V_0 \cup \{ x,z
\})\setminus \{ y \}, V_2\setminus \{ x,z \},V_3 \cup \{ y \})$,
is a DRDF of $G$ of weight less than $w(f)$, and again $(1)$ and
$(2)$ are proved.
\end{proof}

Applying Theorem \ref{thmn1} on the path $xyz$ such that one of the vertices $x,y$ or $z$ has minimum degree, we obtain the next result immediately.

\begin{cor}\label{corn0}
If $G$ is a connected graph of order $n \geq 3$, then
\[
b_{dR}(G) \leq \delta(G)+2\Delta(G)-3.
\]
\end{cor}

\begin{cor}\label{corn1}
If a graph $G$ has a support vertex $v$ of degree at least three such that all of its neighbors except one are leaves, then $b_{dR}(G) \leq 2$.
\end{cor}

\begin{proof}
Let $N(v)= \{ v_1,v_2,...,v_k \} $ such that $\deg_G(v_k) \geq
2$. Applying Theorem \ref{thmn1} (1) on the path $v_1vv_2$ in the
case $\deg_G(v)=k=3$, we obtain $b_{dR}(G) \leq 2$ immediately.
Assume now that $\deg_G(v)=k \geq 4$. Let $f=(V_0,V_2,V_3) $ be a
$\gamma_{dR}$-function of $G-vv_1$. It follows that $v_1 \in V_2$
and, without loss of generality, assume that $v \in V_3$.
Therefore $(V_0 \cup \{v_1 \},V_2\setminus \{ v_1 \},V_3)$ is a
DRDF on $G$ of weight $\gamma_{dR}(G-vv_1)-2$, and thus
$b_{dR}(G)=1$.
\end{proof}

\begin{cor}
For any tree $T$ with at least three vertices, $b_{dR}(T) \leq 2$.
\end{cor}

\begin{proof}
If $T$ has a support vertex $v$ of degree at least three such
that all of its neighbors except one is leaf, then $b_{dR}(T)
\leq 2$ by Corollary \ref{corn1}. So assume that for any support
vertex $v$ either $\deg_T(v)=2$ or $v$ has at least two neighbors
which are not leaves. Let $P=v_1,v_2,...,v_k$ be a longest path
of $T$. By the assumption, $\deg_T(v_2)=2$. If $g$ is a
$\gamma_{dR}(T- \{v_1v_2,v_2v_3 \})$-function, then
$g(v_1)=g(v_2)=2$. Now if we let $h(v_1)=0,h(v_2)=3$ and $h=g$
for other vertices of $T$, then $h$ is a DRDF on $T$ of weight
less than $w(g)$ and so $\gamma_{dR}(T- \{v_1v_2,v_2v_3 \}) >
\gamma_{dR}(T)$. Thus the proof is complete.
\end{proof}

{\bf Problem:} Characterize trees $T$ with $b_{dR}(T) =1$ or $b_{dR}(T) =2$.

We next improve Theorem \ref{thmn1}.
\begin{thm}\label{3.5}
If $G$ is a connected graph of order $n\geq 2$ and  $uv \in E(G)$ then
\[
b_{dR}(G)\leq \deg_G(u) + \deg_G(v)-1-|N(u)\cap N(v)|.
\]
\end{thm}

\begin{proof}
It is not hard to see that $\gamma_{dR}(G- U)>\gamma_{dR}(G)$,
where $U=(\{tu|t\in N(u)\}\cup \{sv|s\in N(v)\}\setminus
\{rv|r\in N(u)\cap N(v)\})$. Note that $|U|=\deg_G(u) +
\deg_G(v)-1-|N(u)\cap N(v)|$.
\end{proof}

\begin{cor}\label{corn2}
If $G$ is a connected graph, then $b_{dR}(G)\leq \Delta(G) + \delta(G) -1$.
\end{cor}

Note that Corollary \ref{corn2} improves Corollary \ref{corn0}. By Corollary \ref{10} we obtain the following improvement of Corollary \ref{corn2} if the graph is planar.

\begin{thm}\label{77}
If $G$ is a connected graph  and  $uwv$ is a path of $G$, then
$b_{dR}(G)\leq \deg_G(u) + \deg_G(v)-1$.
\end{thm}
\begin{proof}
It is not hard to see that $\gamma_{dR}(G- U)>\gamma_{dR}(G)$,
where $U=\{tu|t\in N(u)\}\cup \{sv|s\in (N(v)\setminus \{w\})\}$.
Since $|U|=\deg_G(u) + \deg_G(v)-1$, the proof is complete.
\end{proof}
\begin{rem}
If $G$ is a planar graph and $g\geq 4$, then $\delta(G)\leq 3$. Also
$b_{dR}(G)\leq \Delta(G) + \delta(G) -1\leq \Delta(G) +2$. If $G$ is a
planar graph and $g\geq 6$, then $\delta(G)\leq 2$. Also
$b_{dR}(G)\leq \Delta(G) + \delta(G) -1\leq \Delta(G) +1$.
\end{rem}
\begin{thm}\label{3.9}
If $G$ is a connected planar graph of order $n\geq 2$ without
vertices of degree five, then $b_{dR}(G)\leq 7$.
\end{thm}

\begin{proof}
If $A=\{v\in V(G)|\deg_G(v)\leq 4\}=\{v_1, v_2, \ldots ,v_k\}$,
then Corollary \ref{10} and the hypothesis imply that $A\neq
\emptyset$. Suppose on the contrary that $b_{dR}(G)\geq 8$. In
view of Theorem \ref{77} and the assumption, we deduce that
$d_G(x,y)\geq 3$ for any two distinct vertices $x,y \in A$. Using
Lemma \ref{16}, we define $H_0=G$ and $H_i=H_{i-1}+S_i$ for
$1\leq i\leq k$, where $S_i$ is a subset of $E_{v_i}=\{xy|x,y \in
N(v_i), xy\notin E(H_{i-1})\}$ such that $H_{i-1}+S_i$ is still a
planar graph and $H_i[N(v_i)]$ is $2$-connected when
$\deg_G(v_i)\geq 3$.\\
Now let $x\in A$ and $y\in N_G(x)$. If $\deg_G(x)\leq 2$, then it
follows from the assumption and Theorem \ref{77} that
$\deg_G(y)\geq 7$ and so $\deg_{H_k}(y)\geq 7$. Assume next that
$\deg_G(x)=3$. By the assumption and Theorem \ref{3.5}, we obtain
\[
8\leq \deg_G(x)+\deg_G(y)-|N_G(x)\cap N_G(y)|-1=\deg_G(y)-|N_G(x)\cap N_G(y)|+2.
\]

If $|N_G(x) \cap N_G(y)|\geq 1$, then we deduce that
$\deg_{H_k}(y) \geq \deg_G(y)\geq 7$. In the remaining case
$N_G(x) \cap N_G(y)= \emptyset$, inequality chain leads to
$\deg_G(y)\geq 6$ and thus $\deg_{H_k}(y)\geq 8$. Finally assume
that $\deg_G(x)=4$. By the assumption and Theorem \ref{3.5}, we
obtain
\[
8\leq b_{dR}(G)\leq \deg_G(x)+\deg_G(y)-|N_G(x) \cap N_G(y)|-1=\deg_G(y)-|N_G(x) \cap N_G(y)|+3.
\]
If $|N_G(x) \cap N_G(y)|\geq 2$, then we deduce that $\deg_{H_k}(y)\geq \deg_G(y)\geq 7$. If $|N_G(x) \cap N_G(y)|=1$, then $\deg_G(y)\geq 6$ and so $\deg_{H_k}(y)\geq 7$. In the remaining case $N_G(x) \cap N_G(y)=\emptyset$, we observe that $\deg_G(y)\geq 5$ and thus $\deg_{H_k}(y)\geq 7$.\\
 According to Lemma \ref{16}, the graph $H_k$ is planar. However, since $d_G(x,y)\geq 3$ for any two distinct vertices $x, y\in A$, we observe that $H_k-A$ is a planar graph with minimum degree at
 least 6. This is a contradiction to Corollary \ref{10}, and the proof is complete.
\end{proof}

\begin{thm}
If $G$ is a connected planar graph of order $n\geq 2$, then
$b_{dR}(G)\leq 8$.
\end{thm}

\begin{proof}Let
\begin{align*}
A_3=& \{v\in V(G)|\deg_G(v)\leq 3\},\\
  A_4=& \{v\in V(G)|\deg_G(v)=4\},\\
     A_5=& \{v\in V(G)|\deg_G(v)=5\}.
   \end{align*}

If $A_5=\emptyset$, then Theorem \ref{3.9} implies the desired
result. Thus we can assume $A_5\neq \emptyset$. Suppose on the
contrary that $b_{dR}(G)\geq 9$. In view of Theorem \ref{77} and
the assumption, if $x\in A_3\cup A_4$ and $ y \in A_3\cup A_4\cup
A_5$, then  $d_G(x,y)\geq 3$. In addition, if $x\in A_3$, then
$deg(y)\geq 7$ for all $y\in N_G(x)$, by theorem \ref{3.5}.

Let $I$ be a maximum independent subset of $A_5$. Then
$A_5\subseteq I \cup N(I)$ and  $N(A_4)\cap N(I)=\emptyset$. Now
let $A_4 \cup I=\{v_1, v_2,\ldots , v_k\}$ and  $H=G- A_3$.
Applying Lemma \ref{16}, we define $H_0=H$ and $H_i=H_{i-1}+S_i$
for $1\leq i\leq k$, where $S_i$ is a subset of $E_{v_i}=\{xy|x,y
\in N(v_i), xy\notin E(H_{i-1})\}$ such that $H_{i-1}+S_i$ is
still a planar graph and $H_i[N(v_i)]$ is $2$-connected. We
proceed with the following claims.

{\bf Claim 1.} If $A_4\neq \emptyset$, then $\deg_{H_k}(y)\geq 7$
for each vertex $y\in N_G(A_4)$. 

To see this, assume that $A_4\neq \emptyset$. Let $x \in A_4$ and
$y \in N_G(x)$. Then Theorem \ref{3.5} and the assumption imply
that
\[
9\leq b_{dR}(G)\leq \deg_G(x)+\deg_G(y)-|N_G(x) \cap N_G(y)|-1=\deg_G(y)-|N_G(x) \cap N_G(y)|+3.
\]
If $|N_G(x) \cap N_G(y)|\geq 1$, then we deduce that
$\deg_{H_k}(y)\geq \deg_G(y)\geq 7$. If  $N_G(x) \cap N_G(y)=
\emptyset$, then the inequality chain leads to $\deg_G(y)\geq 6$
and thus $\deg_{H_k}(y)\geq 8$.

{\bf Claim 2.} $\deg_{H_k}(y)\geq 7$ for each vertex $y\in
N_G(I)$.

To see this, let $x \in I$ and $y \in N_G(x)$. Then Theorem
\ref{3.5} and the assumption imply that
\[
9\leq b_{dR}(G)\leq \deg_G(x)+\deg_G(y)-|N_G(x) \cap N_G(y)|-1=\deg_G(y)-|N_G(x) \cap N_G(y)|+4.
\]
If $|N_G(x)\cap N_G(y)|\geq 2$, then we deduce that $\deg_{H_k}(y)\geq \deg_G(y)\geq 7$. If $|N_G(x)\cap N_G(y)|=1$, then $\deg_G(y)\geq 6$ and so $\deg_{H_k}(y)\geq 7$. In the remaining case $N_G(x)\cap N_G(y)=\emptyset$, we observe that $\deg_G(y)\geq 5$ and thus $\deg_{H_k}(y)\geq 7$.\\

Combining Claims 1 and 2, we find that $G^*=H_k-A_4$ is a planar graph with the following properties. The minimum degree of $G^*$ is 5, $I=\{v\in V(G^*)|\deg_{G^*}(v)=5\}$ is an independent 
set in $G^*$ and $\deg_{G^*}(v)\geq 7$ for each vertex $v\in N_{G^*}(I)=N_G(I).$ Let $B$ be the bipartite graph with the bipartition $I$ and $N(I)$ and the edge set $\{uv\in E(G^*)|u\in I, v\in N(I)\}$. Then $B$ is a bipartite planar graph with exactly $5|I|$ edges. Applying 
Theorem \ref{9} with girth $g\geq 4$, we obtain $5|I|\leq 2|I|+2|N(I)|-4$(note that this bound remains valid if $g=\infty$, that means that $B$ is a forest) and therefore $|N(I)|\geq \frac{3}{2}|I|+2$.
Altogether we find

\begin{align*}
  |E(G^*)|&=\frac{1}{2}\sum_{v\in V(G^*)}\deg_{G^*}(v)\\
  &\geq \frac{1}{2}(5|I|+7|N(I)|+6(|V(G^*)|-|I|-|N(I)|))\\
  &=3|V(G^*)|+\frac{1}{2}|N(I)|-\frac{1}{2}|I|\\
  &\geq 3|V(G^*)|+\frac{1}{4}|I|+I>3|V(G^*)|-6
\end{align*}
  This is a contradiction to Corollary \ref{15}, and the proof is complete.
  \end{proof}

\section{Complexity of double Roman bondage number}

In this section, we study the NP-hardness of the double Roman bondage number problem. The decision problem of the double Roman bondage number problem is stated as follows.

{\bf Double Roman bondage number problem (DRBN)}:\\
\textbf{Instance}: A nonempty graph $G$ and a positive integer $k$.\\
\textbf{Question}: Is $b_{dR}(G) \leq k$?\\

We will prove the NP-hardness of the double Roman bondage number problem by transforming from a known NP-complete problem, namely 3-satisfiability problem that is known to be NP-complete \cite{gj}.


At first we recall some terms for the $3$SAT problem. Let $U$ be
a set of Boolean variables. If $u$ is a variable in $U$, then $u$
and $\overline{u}$ are \textit{literals} over $U$. A
\textit{truth assignment} for $U$ is a mapping $t : U \rightarrow
\{T, F\}$.  We call $u$ true under $t$ if $t(u) = T$, otherwise
it is called false under $t$. The literal $u$ is true under $t$
if and only if the variable $u$ is true under $t$; the literal
$\overline{u}$ is true if and only if the variable $u$ is false.
A \textit{clause} over $U$ is a set of literals over $U$. A
clause represents the disjunction of literals and is satisfied by
a truth assignment if and only if at least one of its members is
true under that assignment. A collection $\xi$ of clauses over
$U$ is \textit{satisfiable} if and only if there exists some
truth assignment for $U$ that simultaneously satisfies all the
clauses in $\xi$. Such a truth assignment is called a
\textit{satisfying truth assignment} for $\xi$. The 3-SAT is
specified as follows.
\\

{\bf $3$-satisfiablity problem ($3$-SAT)}:\\
{\bf Instance}: A collection $\xi = \{D_1,D_2, \ldots ,D_m\}$ of clauses over a finite set
$U$ of variables such that $|D_j | = 3$ for $j = 1, 2, \ldots ,m$.\\
{\bf Question}: Is there a truth assignment for $U$ that satisfies all the clauses
in $\xi$?

\begin{thm}
DRBN is NP-hard even for bipartite graphs.
\end{thm}

\begin{proof}
Let $U=\{u_1,u_2,\ldots,u_n\}$ and $\xi=\{D_1, D_2,\ldots, D_m\}$
be an arbitrary instance of $3$-SAT. We construct a bipartite
graph $G$ and a positive integer $k$ such that $\xi$ is
satisfiable if and only if $b_{dR}(G)\leq k$. The graph $G$ is
constructed as follows. For each $i=1,2,\ldots,n$, corresponding
to the variable $u_i\in U$, we associate a graph $H_i$ with
vertex set $V(H_i)=\{u_i, \overline{u}_i,w_i, v_i, v^\prime_i,
x_i, y_i, z_i\}$ and edge set
\[
E(H_i)=\{u_iz_i,u_iv_i,v_iw_i, \overline{u}_iz_i, \overline{u}_iv^\prime_i, v^\prime_iw_i,w_iz_i, y_iv_i, y_iv^\prime_i, y_iz_i, x_iv_i, x_iv^\prime_i\}.
\]
For each $j=1,2,\ldots,m$, corresponding to the clause $D_j=\{p_j, q_j, r_j\}\in \xi $, associate a single vertex $c_j$ and add edge set $E_j=\{c_jp_j, c_jq_j, c_jr_j\}, 1\leq j \leq m$. Next, add a cycle
$C_8=l_1l_2l_3l_4l_5l_6l_7l_8l_1$, and join $l_2$ and $l_4$  to each vertex $c_j$ with $1\leq j \leq m$. Finally add a new vertex $l_9$ and join it to both $l_1$ and $l_5$, and set $k=1$.
Figure 1 shows an example of the graph $G$ when $U=\{u_1,u_2,u_3,u_4\}$ and $\xi=\{D_1, D_2, D_3\}$, where $D_1=\{u_1,\overline{u_2},u_4\}$, $D_2=\{\overline{u_1},\overline{u_2}, u_4\}$
and $D_3=\{u_2,u_3,\overline{u_4}\}$.

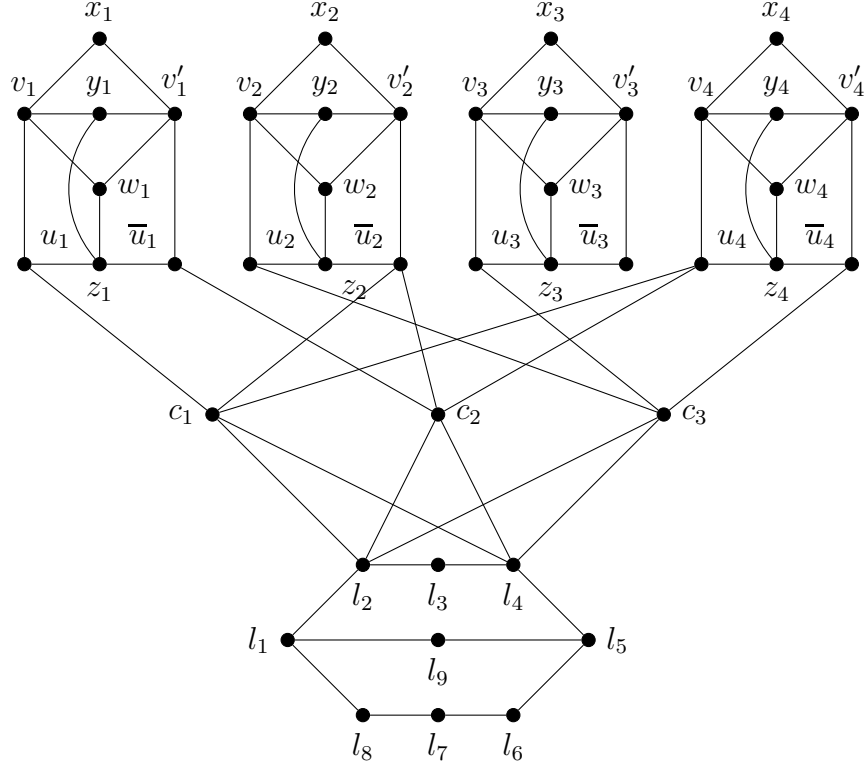
\tikzstyle{vertex}=[circle, draw, inner sep=0pt, minimum size=5pt]
\newcommand{\vertex}{\node[vertex]}

\newcommand{\vertexs}{\node[vertexs]}
\[
\begin{tikzpicture}
\vertex[fill] (a1) at (0,0) [label=above right:$u_{1}$]{};
\vertex [fill](a2) at (1,0) [label=below:$z_{1}$]{};
\vertex [fill](a3) at (2,0) [label=above left:$\overline{u}_{1}$]{};
\vertex [fill](a9) at (1,2) [label=above :$y_{1}$]{};
\vertex [fill](a7) at (0,2) [label=above:$v_{1}$]{};
\vertex [fill](a5) at (2,2) [label=above:$v^{\prime}_{1}$]{};
\vertex [fill](a6) at (1,3) [label=above:$x_{1}$]{};
\vertex [fill](a4) at (1,1) [label=right:$w_{1}$]{};
\vertex [fill](b1) at (2.5,-2) [label= left:$c_{1}$]{};
\vertex[fill] (e1) at (3,0) [label=above right:$u_{2}$]{};
\vertex [fill](e2) at (4,0) [label=below right:$z_{2}$]{};
\vertex [fill](e3) at (5,0) [label=above left:$\overline{u}_{2}$]{};
\vertex [fill](e9) at (4,2) [label= above:$y_{2}$]{};
\vertex [fill](e7) at (3,2) [label=above:$v_{2}$]{};
\vertex[fill] (e5) at (5,2) [label=above:$v^{\prime}_{2}$]{};
\vertex [fill](e6) at (4,3) [label=above:$x_{2}$]{};
\vertex[fill] (f1) at (5.5,-2) [label= right:$c_{2}$]{};
\vertex [fill](e4) at (4,1) [label=right:$w_{2}$]{};
\vertex[fill] (g1) at (6,0) [label=above right:$u_{3}$]{};
\vertex [fill](g2) at (7,0) [label=below:$z_{3}$]{};
\vertex[fill] (g3) at (8,0) [label=above left:$\overline{u}_{3}$]{};
\vertex[fill] (g9) at (7,2) [label=above:$y_{3}$]{};
\vertex[fill] (g7) at (6,2) [label=above:$v_{3}$]{};
\vertex [fill](g5) at (8,2) [label=above:$v^{\prime}_{3}$]{};
\vertex [fill](g6) at (7,3) [label=above:$x_{3}$]{};
\vertex [fill](h1) at (8.5,-2) [label= right:$c_{3}$]{};
\vertex [fill](g4) at (7,1) [label=right:$w_{3}$]{};
\vertex[fill] (i1) at (9,0) [label=above right:$u_{4}$]{};
\vertex [fill](i2) at (10,0) [label=below:$z_{4}$]{};
\vertex [fill](i3) at (11,0) [label=above left:$\overline{u}_{4}$]{};
\vertex [fill](i9) at (10,2) [label=above:$y_{4}$]{};
\vertex [fill](i7) at (9,2) [label=above:$v_{4}$]{};
\vertex [fill](i5) at (11,2) [label=above:$v^{\prime}_{4}$]{};
\vertex [fill](i6) at (10,3) [label=above:$x_{4}$]{};
\vertex [fill](i4) at (10,1) [label=right:$w_{4}$]{};
\vertex[fill] (j1) at (3.5,-5) [label=left:$l_{1}$]{};
\vertex[fill] (j2) at (4.5,-4) [label=below:$l_{2}$]{};
\vertex[fill] (j3) at (5.5,-4) [label=below:$l_{3}$]{};
\vertex[fill] (j4) at (6.5,-4) [label=below:$l_{4}$]{};
\vertex[fill] (j5) at (7.5,-5) [label=right:$l_{5}$]{};
\vertex[fill] (j6) at (6.5,-6) [label=below:$l_{6}$]{};
\vertex[fill] (j7) at (5.5,-6) [label=below:$l_{7}$]{};
\vertex[fill] (j8) at (4.5,-6) [label=below:$l_{8}$]{};
\vertex[fill] (j9) at (5.5,-5) [label=below:$l_{9}$]{};
\path
(a2) edge[bend left=40] (a9)
(a1) edge (a2)
(a2) edge (a3)
(a5) edge (a6)
(a6) edge (a7)
(a5) edge (a9)
(a1) edge (a7)
(a3) edge (a5)
(a7) edge (a9)
(a4) edge (a7)
(a4) edge (a5)
(e1) edge (e2)
(e2) edge (e3)
(e5) edge (e6)
(e6) edge (e7)
(e5) edge (e9)
(e3) edge (e5)
(e7) edge (e9)
(e4) edge (e7)
(e4) edge (e5)
(e2) edge[bend left=40] (e9)
(g1) edge (g2)
(g2) edge (g3)
(g5) edge (g6)
(g6) edge (g7)
(g5) edge (g9)
(g3) edge (g5)
(g7) edge (g9)
(g7) edge (g1)
(e7) edge (e1)
(g2) edge (g4)
(a2) edge (a4)
(e2) edge (e4)
(g4) edge (g7)
(g4) edge (g5)
(g2) edge[bend left=40] (g9)
(i1) edge (i2)
(i2) edge (i3)
(i5) edge (i6)
(i6) edge (i7)
(i5) edge (i9)
(i3) edge (i5)
(i7) edge (i9)
(i7) edge (i1)
(i2) edge (i4)
(i4) edge (i7)
(i4) edge (i5)
(i2) edge[bend left=40] (i9)
(j1) edge (j2)
(j2) edge (j3)
(j3) edge (j4)
(j4) edge (j5)
(j5) edge (j6)
(j6) edge (j7)
(j7) edge (j8)
(j8) edge (j1)
(j1) edge (j9)
(j5) edge (j9)
(b1) edge (j2)
(b1) edge (j4)
(f1) edge (j2)
(f1) edge (j4)
(h1) edge (j2)
(h1) edge (j4)
(b1) edge (a1)
(b1) edge (e3)
(b1) edge (i1)
(f1) edge (a3)
(f1) edge (e3)
(f1) edge (i1)
(h1) edge (e1)
(h1) edge (g1)
(h1) edge (i3);
\end{tikzpicture}
\]

It is easy to see that $G$ is a bipartite graph, and the
construction can be accomplished in polynomial time. Let $k=1$.
We will prove that $\xi$ is satisfiable if and only if
$b_{dR}(G)=1$. We proceed with a series of claims namely Claim 1,
Claim 2, Claim 3 and Claim 4.

{\bf Claim 1.} $\gamma_{dR}(G)\geq 6n + 8$. If equality hold, then for any
$\gamma_{dR}$-function $f$ on $G$,\\
\item[$i)$] $f(H_i) = 6$, for any $i=1,\ldots ,n,$  \\
\item[$ii)$]$|\{u_i, \overline{u}_i\} \cap V_3|\leq1$, for any $i=1,\ldots ,n,$  \\
\item[$iii)$]$\{u_i, \overline{u}_i\} \cap V_2=\emptyset$, for any $i=1,\ldots ,n,$ \\
\item[$iv)$]$f(l_1) =f(l_3)=f(l_5) =f(l_7)=2$ and $\{c_j: j=1,2, \ldots, m\}\subseteq V_0$.

{\bf Proof of Claim 1.} Let $f$ be a $\gamma_{dR}$-function of $G$, and $i\in \{1,2,...,n\}$.\\
If $f(u_i)\geq 2$ and $f(\overline{u}_i)\geq 2$, then  $\sum_{v\in N[x_i]} f(v)\geq 2$ and so $\sum_{v\in V(H_i)} f(v)\geq 6$. If $f(u_i)=0$ or $f(\overline{u}_i)=0$, then by considering $H_i-u_i-\overline{u}_i$ and $H_i-u_i$ or $H_i-\overline{u}_i$ and Theorem \ref{12}, we obtain that $\sum_{v\in V(H_i)} f(v)\geq 6$.
If $f(u_i)=2$ and $f(\overline{u}_i)=0$ or $f(u_i)=3$ and $f(\overline{u}_i)=0$, then similarly we observe that $\sum_{v\in V(H_i)} f(v)\geq 6$. On the other hand $\sum_{i=1}^9 f(l_i)\geq 8$.
Consequently, $\gamma_{dR}(G)\geq 6n + 8$.
Assume that $\gamma_{dR}(G)=6n + 8$. Then $\sum_{i=1}^9 f(l_i)= 8$ and $f(H_i)=\sum_{v\in V(H_i)} f(v)= 6$ for $i=1,2,...,n$. Suppose that there is an integer $j\in \{1,2...,n\}$ such that
$f(u_j)=f(\overline{u}_j)=3$. Since $\sum_{v\in N[x_j]}f(v)\geq 2$, and $\sum_{v\in N[w_j]}f(v)\geq 2$, we find that $f(H_j)> 6$, a contradiction. Thus $|\{u_i, \overline{u}_i\} \cap V_3|\leq 1$, for
any $i=1,\ldots ,n$. Parts $(iii)$ and $(iv)$ are proved similarly. Thus the proof of Claim 1 is complete. $\diamondsuit$

{\bf Claim 2.} $\gamma_{dR}(G) = 6n + 8$ if and only if $\xi$ is satisfiable.

{\bf Proof of Claim 2.} Assume that $\gamma_{dR}(G) = 6n + 8$ and
let $f$ be a $\gamma_{dR}(G)$-function. By Claim 1, at most one
of $f(u_i)$ and $f(\overline{u}_i)$ is equal to $3$ for each
$i=1,2,...,n$. Define a mapping $t : U \rightarrow \{T, F\}$ by
\begin{center}
$$ t(u_i)=
\begin{cases}
T &\text{if $f(u_i)=3$ or $f(u_i)\neq 3$ and $f(\overline{u}_i)\neq 3$},\\
F &\text{if $f(\overline{u}_i)=3$}.
\end{cases}$$
\end{center}

We now show that $t$ is a satisfying truth assignment for $\xi$.
It is sufficient to show that every clause in $\xi$ is satisfied
by $t$. To this end, we arbitrarily choose a clause $D_j\in \xi$
with $1 \leq j \leq m$. By Claim 1, $f(c_j)=f(l_2)=f(l_4)=0$.
Thus there exists some $i$ with $1\leq i \leq n$ such that
$f(u_i) = 3$ or $f(\overline{u}_i) = 3$, where $c_j$ is adjacent
to $u_i$ or $\overline{u}_i$. Assume that $c_j$ is adjacent to
$u_i$, where $f(u_i) = 3$. Since $u_i$ is adjacent to $c_j$ in
$G$, the literal $u_i$ is in the clause $D_j$ by the construction
of $G$. Since $f(u_i) = 3$, it follows that $t(u_i) = T$ which
implies that the clause $D_j$ is satisfied by $t$. Next assume
that $c_j$ is adjacent to $\overline{u}_i$, where
$f(\overline{u}_i) = 3$. Since $\overline{u}_i$ is adjacent to
$c_j$ in $G$, the literal $\overline{u}_i$ is in the clause
$D_j$. Since $f(\overline{u}_i) = 3$, it follows that $t(u_i) =
F$. Thus, $t$ assigns $\overline{u}_i$ the truth value $T$, that
is, $t$ satisfies the clause $D_j$. Since $j$ is chosen
arbitrarily, thus $t$ satisfies all the clauses in $\xi$.
Consequently, $\xi$ is satisfiable.

Conversely, assume that $\xi$ is satisfiable, and let $t : U
\rightarrow \{T, F\}$ be a satisfying truth assignment for $\xi$.
Create a function $f$ on $V (G)$ as follows: if $t(u_i) = T$,
then let $f(u_i) = f(v^\prime)=3$, and if $t(u_i) = F$, then let
$f(\overline{u}_i) = f(v_i) = 3$. Let $f(l_1) =f(l_3)=f(l_5)
=f(l_7)=2$. Clearly, $f(V(G)) = 6n + 8$. Since $t$ is a satisfying
truth assignment for $\xi$, at least one of literals in $D_j$ is
true under the assignment $t$, for each $j = 1, 2, \ldots ,m$. It
follows that the corresponding vertex $c_j$ in $G$ is adjacent to
at least one vertex $w$ with $f(w) = 3$, since $c_j$ is adjacent
to each literal in $D_j$ by the construction of $G$. Thus $f$ is
a DRDF of $G$, and so $\gamma_{dR}(G)\leq f(G)= 6n + 8$. Since by
Claim 1, $\gamma_{dR}(G)\geq 6n + 8$, we conclude that
$\gamma_{dR}(G)= 6n + 8$. Thus the proof of Claim 2 is complete.
$\diamondsuit$

{\bf Claim 3.} $\gamma_{dR}(G-e)\leq 6n + 9$ for any $e \in E(G)$.

{\bf Proof of Claim 3.} For any edge $e \in E(G)$, it is
sufficient to construct a DRDF $f$ on $G-e$ with weight $4n+9$.
We first assume that $e \in
\{l_1l_2,l_1l_8,l_1l_9,l_3l_4,l_6l_7\}$ or $e=c_jl_4$, $c_ju_i$
or $e = c_j \overline{u}_i$, for some $j = 1, 2, \ldots ,m$, and
$i=1, 2, \ldots , n$. We define a function $f$ by $f(l_2) =
f(l_5) = f(l_8)=3 $, $f(u_i) = f(v^\prime_i) = 3$ for each $i =
1, 2, \ldots , n$, and $f(v)=0$ otherwise. Then $f$ is a DRDF of
$G-e$ with $f(G-e) = 6n + 9$. If  $e \in
\{l_5l_4,l_5l_6,l_5l_9,l_2l_3,l_7l_8\}$, then we define $f$ by
$f(l_1) = f(l_4) = f(l_6)=3 $, $f(u_i) = f(v^\prime_i) = 3$ for
each $i = 1, 2, \ldots , n$, and $f(v)=0$ otherwise. Then $f$ is
a DRDF of $G-e$ with $f(G-e) = 6n + 9$. If $e$ is not incident
with $u_i$ or $v^\prime_i$ for each $i$, then we define $f$
$f(l_1) = f(l_4) = f(l_6)=3 $ and $f(u_i) = f(v^\prime_i) = 3$,
and $f(v)=0$ otherwise. If $e$ is not incident with
$\overline{u}_i$ or $v_i$, then we define $f$ by $f(l_1) = f(l_4)
= f(l_6)=3 $, $f(\overline{u}_i) = f(v_i) = 3$, and $f(v)=0$
otherwise. If $e = u_iv_i$ or $\overline{u}_iv^\prime_i$, for
some $i$, then we define $f$ by $f(l_1) = f(l_4) = f(l_6)=3 $,
$f(x_i) = f(z_i) = 3$ for each $i = 1, 2, \ldots , n$, and
$f(v)=0$ otherwise. Then $f$ is a DRDF of $G-e$ with $f(G-e) = 6n
+ 9$ and thus $\gamma_{dR}(G-e)\leq 6n + 9$. $\diamondsuit$

{\bf Claim 4.}
$\gamma_{dR}(G)= 6n + 8$ if and only if $b_{dR}(G) = 1$.

{\bf Proof of Claim 4.}
Assume $\gamma_{dR}(G)= 6n + 8$ and consider the edge $e = l_1l_2$. Suppose
$\gamma_{dR}(G)= \gamma_{dR}(G-e)$. Let $f^\prime$ be a
$\gamma_{dR}$-function of $G - e$. It is clear that $f^\prime$ is also
a $\gamma_{dR}$-function on $G$. By Claim 1, we have $f^\prime(c_j) = 0$ for each $j = 1, 2, \ldots ,m$
and $f(l_2) =f(l_4)=f(l_6) =f(l_8)=f(l_9)=0$. But then $f^\prime(N[l_2]) =2$, a contradiction. Hence,
$\gamma_{dR}(G)< \gamma_{dR}(G-e)$, and so $b_{dR}(G) = 1$.\\
Conversely, assume that $b_{dR}(G) = 1$. By Claim 1, we have
$\gamma_{dR}(G)\geq 6n + 8$. Let $e^\prime$ be an edge such that
$\gamma_{dR}(G)< \gamma_{dR}(G-e^\prime)$. By Claim 3, we have
$\gamma_{dR}(G-e^\prime)\leq 6n+8$. Thus, $6n+8\leq
\gamma_{dR}(G)< \gamma_{dR}(G-e^\prime)\leq 6n+9$, which yields
$\gamma_{dR}(G)= 6n + 8$. $\diamondsuit$

By Claims 2 and 4, $b_{dR}(G) = 1$ if and only if there is a truth
assignment for $U$ that satisfies all clauses in $\xi$. Since the
construction of the double Roman bondage number instance is
straightforward from a $3$-satisfiability instance, the size of
the double Roman bondage number instance is bounded above by a
polynomial function of the size of $3$-satisfiability instance. It
follows that this is a polynomial reduction and the proof is
complete.
\end{proof}


\begin{thebibliography}{99}
\bibitem{acs}
H. Abdollahzadeh Ahangar, M. Chellali and S.M. Sheikholeslami, On
the double Roman domination in graphs, \textit{Discrete Appl.
Math}. 103 (2017) 245--258.

\bibitem{akq} S. Akbari, M. Khatirinejad and S. Qajar, A note on Roman
bondage number of planar graphs, \textit{Graphs Combin.} 29
(2013), 327--331.


\bibitem{bhsx} A. Bahremandpour, F.-T. Hu, S.M. Sheikholeslami and J.-M. Xu,
On the Roman bondage number of a graph, \textit{Discrete Math.
Alg. Appl.} 5 (1) (2013), 1350001 (15 pages).

\bibitem{bhns} D. Bauer, F. Harary, J. Nieminen and C.L. Suffel,
Domination alteration sets in graphs, \textit{Discrete Math.} 47
(1983), 153--161.

\bibitem{bhh} R.A.  Beeler,  T.W.  Haynes  and  S.T.  Hedetniemi,
Double Roman domination, \textit{Discrete Appl. Math.} 211
(2016), 23--29.


\bibitem{cdhh}E.J. Cockayane, P.M. Dreyer Jr., S.M. Hedetniemi and S.T. Hedetniemi, On Roman domination in graphs, \textit{Discrete Math.} 278 (2004) 11-22.

\bibitem{dhtv} J.E. Dunbar, T.W. Haynes, U. Teschner and L. Volkmann, \textit{Bondage,
insensitivity, and reinforcement}, in: T.W. Haynes, S.T.
Hedetniemi and P.J. Slater (Eds.), Domination in Graphs: Advanced
Topics (Marcel Dekker, Ne w York, 1998), 471--489.

\bibitem{fjkr} J.F. Fink, M.S. Jacobson, L.F. Kinch and J. Roberts, The
bondage number of a graph, \textit{Discrete Math.} 86 (1990),
47--57.

\bibitem{gj} M.R. Garey and D.S. Johnson, Computers and Intractability: A Guide to the Theory of NP-Completeness, Freeman, San Francisco,
1979.

\bibitem{hhs}T. W. Haynes, S. T. Hedetniemi and P. J. Slater, \textit{Fundamentals of Domination in Graphs},
Marcel Dekker, Inc. New York, 1998.

\bibitem{hr} B.L. Hartnell and D.F. Rall, Bounds on the bondage number of a
graph, \textit{Discrete Math.} 128 (1994), 173--177.

\bibitem{jr} N. Jafari Rad and H. Rahbani,
Some progress on double Roman domination in graphs,
\textit{Discuss. Math. Graph Theory} (to apeear).

\bibitem{jv}  N. Jafari Rad and L. Volkmann, Roman bondage in graphs,
\textit{Discuss. Math. Graph Theory} 31 (4) (2011), 763--773.

\bibitem{rr} C. S. ReVelle and K. E. Rosing, Defendens imperium Romanum: a classical
problem in military strategy, \textit{Amer Math. Monthly} 107
(2000), 585--594.

\bibitem{s} I. Stewart, Defend the roman empire!, \textit{Sci. Amer.} 281 (6) (1999), 136--139.

\bibitem{x} J.-M. Xu, On bondage numbers of graphs  a survey with some
comments, \textit{International J. Combin.} (2013), Article ID
595210, 34 pages, 2013. doi:10.1155/2013/5952.

\bibitem{v} L. Volkmann, Double Roman domination and domatic numbers of
graphs, \textit{Comm. Combin. Optim.} 3 (2018), 71--77.

\bibitem{w} D.B. West, \textit{Introduction to Graph Theory}, 2nd Edition, Prentice-Hall, Inc. (2001).

\end{thebibliography}
\end{document}